\documentclass[11pt,oneside,article]{memoir}

\settrims{0pt}{0pt} 
\settypeblocksize{*}{34pc}{*} 
\setlrmargins{*}{*}{1} 
\setulmarginsandblock{1in}{1in}{*} 
\setheadfoot{\onelineskip}{2\onelineskip} 
\setheaderspaces{*}{1.5\onelineskip}{*} 
\checkandfixthelayout

\usepackage{graphicx}
\usepackage{amssymb}
\usepackage{amsfonts}
\usepackage{amsthm}
\usepackage{mathtools}
\usepackage{dsfont}
\usepackage{subfig}
\usepackage{etex}
\usepackage{mathrsfs}
\usepackage[sc]{mathpazo}
\usepackage{url}

\usepackage[usenames,dvipsnames]{xcolor}
\usepackage{tikz}
\usetikzlibrary{arrows,calc,positioning,scopes,cd,decorations.markings,fit}

\RequirePackage{hyperref}
\hypersetup{colorlinks,linkcolor={blue!50!black},citecolor={red!50!black},urlcolor={blue!80!black}}

\DeclareMathOperator{\id}{id}

\DeclareMathOperator{\Hom}{Hom}
\DeclareMathOperator{\Mor}{Mor}

\DeclareMathOperator{\Ob}{Ob}

\newtheorem{theorem}{Theorem}[chapter]

\newtheorem{proposition}[theorem]{Proposition}

\theoremstyle{remark}
\newtheorem{remark}[theorem]{Remark}

\theoremstyle{definition}
\newtheorem{definition}[theorem]{Definition}
\newtheorem*{notation*}{Notation}

\theoremstyle{definition}
\newtheorem{example}[theorem]{Example}

\newcommand{\cat}[1]{\mathscr{#1}}
\newcommand{\Cat}[1]{\mathsf{#1}}
\newcommand{\fun}[1]{\mathcal{#1}}

\def\too{\longrightarrow}
\def\Set{\Cat{Set}}

\def\FinSet{\Cat{FinSet}}

\newcommand{\To}[1]{\xrightarrow{#1}}
\def\iso{\cong}
\def\ss{\subseteq}
\def\TFS{\Cat{TFS}}
\def\dom{\tn{dom}}
\def\cod{\tn{cod}}
\def\op{^{\text{op}}}

\def\ZZ{\mathbb Z}

\newcommand{\ol}[1]{\overline{#1}}

\newcommand{\inp}[1]{#1^{\tn{in}}}
\newcommand{\outp}[1]{#1^{\tn{out}}}
\newcommand{\vinp}[1]{\ol{\inp{#1}}}
\newcommand{\voutp}[1]{\ol{\outp{#1}}}

\def\tn{\textnormal}

\def\|{\;|\;}
\def\m1{{-1}}
\def\singleton{{\{\ast\}}}

\tikzset{
	wiring diagram/.style={
		every to/.style={out=0,in=180,draw},
		label/.style={
			font=\everymath\expandafter{\the\everymath\scriptstyle},
			inner sep=0pt,
			node distance=2pt and -2pt},
		semithick,
		node distance=1 and 1,
		decoration={markings, mark=at position .5 with {\arrow{stealth};}},
		ar/.style={postaction={decorate}},
		execute at begin picture={\tikzset{
			x=\bbx, y=\bby,
			every fit/.style={inner xsep=\bbx, inner ysep=\bby}}}
		},
	bbx/.store in=\bbx,
	bbx = 1.5cm,
	bby/.store in=\bby,
	bby = 1.75ex,
	bb port sep/.store in=\bbportsep,
	bb port sep=2,
	bb port length/.store in=\bbportlen,
	bb port length=4pt,
	bb min width/.store in=\bbminwidth,
	bb min width=1cm,
	bb rounded corners/.store in=\bbcorners,
	bb rounded corners=2pt,
	bb small/.style={bb port sep=1, bb port length=2.5pt, bbx=.4cm, bb min width=.4cm, bby=.7ex},
	bbthick/.code n args={4}{
		\pgfmathsetlengthmacro{\bbheight}{\bbportsep * (max(#1,#2)+1) * \bby}
		\pgfkeysalso{draw,minimum height=\bbheight,minimum width=\bbminwidth,outer sep=0pt,
			rounded corners=\bbcorners,thick,
			prefix after command={\pgfextra{\let\fixname\tikzlastnode}},
			append after command={\pgfextra{
			\draw[#3]
				\ifnum #1=0{} \else foreach \i in {1,...,#1} {
					($(\fixname.north west)!{\i/(#1+1)}!(\fixname.south west)$) +(-\bbportlen,0) coordinate (\fixname_in\i) -- +(\bbportlen,0) coordinate (\fixname_in\i')}\fi;
			\draw[#4]
				\ifnum #2=0{} \else foreach \i in {1,...,#2} {
					($(\fixname.north east)!{\i/(#2+1)}!(\fixname.south east)$) +(-\bbportlen,0) coordinate (\fixname_out\i') -- +(\bbportlen,0) coordinate (\fixname_out\i)}\fi;
			}}}
	},
	bb/.code 2 args={\pgfkeysalso{bbthick={#1}{#2}{thin}{thin}}},
	bb name/.style={append after command={\pgfextra{\node[anchor=north] at (\fixname.north) {#1};}}}
}

\date{\vspace{-.3in}}
\begin{document}
\tightlists
\firmlists

\title{Nesting of dynamical systems and mode-dependent networks}

\author {David I. Spivak\thanks{This project was supported by ONR grant N000141310260, AFOSR grant FA9550-14-1-0031, and NASA grant NNH13ZEA001N-SSAT}
\\\texttt{\small dspivak@math.mit.edu}
\and
Joshua Tan
\\\texttt{\small joshua.z.tan@gmail.com}
}

\maketitle

\begin{abstract}
For many networks, the connection pattern (often called the topology) can vary in time, depending on the changing state of the modules within the network. This paper addresses the issue of nesting such mode-dependent networks, in which a local network can be abstracted as a single module in a larger network. Each module in the network represents a dynamical system, whose behavior includes repeatedly updating its communicative mode, and these mode in turn dictate the connection pattern. It is in this way that the dynamics of the modules controls the topology of the networks at all levels. This paper provides a formal semantics, using the category-theoretic framework of operads and their algebras, to capture the nesting property and dynamics of mode-dependent networks. We provide a detailed running example to ground the mathematics.
\end{abstract}


\chapter{Introduction}
\label{sec:intro}


If we find that a story or structure repeats itself at various scales of a model, it is often a useful exercise to formalize the model using \emph{operads}, because doing so constrains the model to be highly self-consistent. There is an operad of complete sentences: ``Here is a sentence, here is another sentence, and this entire quotation is a sentence.'' (There is (or may (also) be) an operad of parentheticals.) More importantly, there is an operad of networks, one that naturally models networks of networks.

To say that there is an operad of $X$ is to say that there is a hierarchical, modular theory of $X$. In previous work \cite{RS, S, VSL}, operads that describe special sorts of networks, called \emph{wiring diagrams}, were given and shown to model databases relations, digital circuits, and open continuous-time dynamical systems, among other examples. A wiring diagram is a fixed, graph-like arrangement of nodes and directed edges where (1) nodes represent boxes $\begin{tikzpicture}[wiring diagram,bb small]\node[bb={1}{1}]{};\end{tikzpicture}$ with differentiated input and output ports and (2) one can ``cut-and-paste'' graphs into or out of other graphs by expanding or contracting at a given node. In this note, we extend wiring diagrams to the case where the network topology is not fixed in time but instead varies with respect to the states of various nodes in the wiring diagram. The states of the nodes define a local ``mode'', so we have dubbed these diagrams \emph{mode-dependent networks}.

\paragraph{Motivating example.} Many examples of dynamical processes on complex networks come from biology and living systems, and in such systems the network topology often changes over time as the state evolves; changes in topology in turn influence the state. One classic example is a (human) neural network, where connections are constantly formed or pruned as neurons fire. We present here a network model of the visual system that composes seamlessly from the firing dynamics of individual neurons up to the large-scale dynamics of brain models that simulate blinking, light adaptation, and visual processing.\footnote{There is an ulterior motive for this example; we believe that operads are useful not only in the analysis of complex networks but also in the \emph{design} of complex (scientific) models, such as for large-scale brain models. This example is a demonstration of both uses.}

\begin{example}Consider the diagram below.

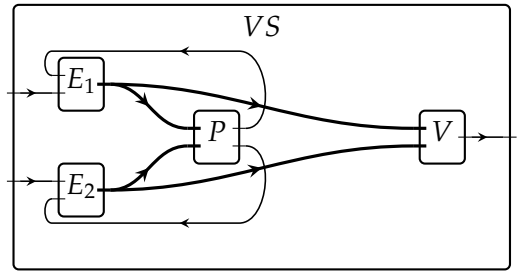
\begin{figure}[h!]
\centering
\[
	\begin{tikzpicture}[wiring diagram,bb port sep=1, bb port length=2.5pt, bbx=.6cm, bb min width=.6cm, bby=1.3ex]
		\node[bbthick={2}{1}{thin}{very thick}, bb name=${E_1}$] (E1) {};
		\node[bbthick={2}{1}{thin}{very thick}, below= 3 of E1, bb name=${E_2}$] (E2) {};
		\node[bbthick={2}{2}{very thick}{thin}, below right = 0 and 2 of E1, bb name=${P}$] (P) {};
		\node[bbthick={2}{1}{very thick}{thin}, right= 4 of P, bb name=$V$] (V) {};
		\node[bb={2}{1}, fit={($(E2.south west)+(0,-2)$) (P) (V) ($(E1.north)+(0,2)$)},bb name =$VS$] (VS) {};
		\draw[ar, very thick] (E1_out1) to (P_in1);
		\draw[ar, very thick] (E2_out1) to (P_in2);
		\draw[ar] let \p1=(P.north east), \p2=(E1.north west), \n1={\y1+\bby}, \n2=\bbportlen in (P_out1) to[in=0] (\x1,\y2+\n2) -- (\x2-\n2,\y2+\n2) to[out=180] (E1_in1);
		\draw[ar] let \p1=(P.south east), \p2=(E2.south west), \n1={\y1+\bby}, \n2=\bbportlen in (P_out2) to[in=0] (\x1,\y2-\n2) -- (\x2-\n2,\y2-\n2) to[out=180] (E2_in2);
		\draw[ar] (VS_in1') to (E1_in2);
		\draw[ar] (VS_in2') to (E2_in1);
		\draw[ar, very thick] (E1_out1) to (V_in1);
		\draw[ar, very thick] (E2_out1) to (V_in2);
		\draw[ar] (V_out1) to (VS_out1');
	\end{tikzpicture}
\]
\caption{A discrete dynamical system of the visual pathway, interpreted over a mode-dependent network in $\cat{O}_\Cat{MDN}$, the operad of mode-dependent networks. The thicker wires represent many parallel connections.}
\label{fig:firsteye}
\end{figure}

The visual pathway from retina to cortex is a complex neurological system that combines autonomic behaviors (blinking, saccading) with layers of image processing in the visual cortex. In this diagram, $E_1$ and $E_2$ represent the eyes and eyelids, $P$ represents a collection of nuclei in the brainstem that regulate the corneal (blink) reflex, and $V$ represents (a slice of) the visual cortex. $VS$ represents the entire visual system. This setup can be described using an operad because the outer box is of the same nature as the inner boxes---i.e., the visual pathway is itself a neural network which may be placed within a larger neural network---and the process of assembly can repeat ad infinitum. In this way, one can recursively build up networks of networks. The goal of this paper is to show how operads can model not only networks of networks but also complex, heterogeneous dynamics over such networks. 

The mode-dependency of Figure~\ref{fig:firsteye} may not be obvious at first glance. We will see it in later examples. This follows a common motif in networked systems; changes in connection pattern are not typically visible in a macro analysis of networks of networks, yet small changes in connectivity can drastically change the dynamics of the system as a whole, for example in power grids \cite{Buld}, the Internet \cite{Albert}, and functional brain networks \cite{Reis}. To see the structure, one must zoom in.
\end{example}

Similar stories of modularity and hierarchy have been considered before, including networks of networks \cite{Reis}, multilayer networks \cite{Kiv, Buld}, and tensor representations thereof. Operads can be used to enrich these graph-theoretic models; operads provide a \emph{formula} for zooming in and out of a complex network, and the proofs in this paper show that the formula for rewiring and assembly is consistent and independent of one's choice of how and where to zoom. This requirement puts fairly strict constraints on the formalism, as mentioned above. Not any seemingly-workable definition of mode-dependent dynamical systems will actually satisfy the nesting property.


All this will be made formal below. We refer the reader to Mac Lane \cite{Mac}, Awodey \cite{Awo}, or Spivak \cite{CT4S} for background on category theory (in decreasing order of difficulty), to \cite{Lei} for specific background on operads, algebras, and monoidal categories, and to \cite{field}, \cite{Newman}, or \cite{Kiv} for background on complex networks and multilayer networks. Other category-theoretic approaches to networks and their dynamics include \cite{DL1, DL2}, \cite{BE, BF}, and \cite{Sco}.

In Section~\ref{sec:background} we will cover the prerequisites in category theory, before giving the precise definition of mode-dependent networks in Section~\ref{sec:MDN} and of modal dynamical systems in Section~\ref{sec:DS_MDN}. Readers interested mostly in applications may skip directly to Section~\ref{sec:applications}, though we recommend at least skimming some of the examples in Section~\ref{sec:DS_MDN}.


\chapter{Background}\label{sec:background}
We begin with some notation and basic terms from category theory.

\begin{notation*}


Let $\Set$ denote the category of sets and functions between them. Let $\Cat{FinSet}\ss\Set$ denote the full subcategory spanned by the finite sets. If $A,B\in\Ob\cat{C}$ are objects of a category, we may denote the set of morphisms between them either by $\Hom_{\cat{C}}(A,B)$ or by $\cat{C}(A,B)$. If $A\in\Ob\cat{C}$ is an object, we may denote the identity morphism on $A$ either by $\id_A$ or simply by $A$. If there is a unique element in $\cat{C}(A,B)$, we may denote it $!\colon A\to B$. For example, there is a unique function $!\colon\emptyset\to A$ for any set $A$. The symbol $\emptyset$ represents the empty set. For any category, let $\Mor \cat{C}$ be the union of all the hom-sets, $\Mor \cat{C} = \sqcup_{A, B \in \cat{C}} \Hom(A,B)$. There are maps $\dom, \cod: \Mor \cat{C} \to \Ob \cat{C}$ sending a morphism $f$ to its domain and codomain, respectively.

\end{notation*}

Recall that wiring diagrams are made up of boxes with a notion of input and output port, each labeled by the type of data that goes through that port. The following definitions are used to make this precise.

\begin{definition}\label{def:typed finite sets} 
The category of \emph{typed finite sets}, denoted $\TFS$, is defined as follows. An object in $\TFS$ is a finite set of sets, 
\[\Ob\TFS\coloneqq\{(A,\tau)\; |\; A\in\Ob\FinSet,\; \tau\colon A\to\Set)\}.\] 
 We call $\tau$ the \emph{typing function}, and for any element $a\in A$, we call the set $\tau(a)$ its {\em type}. If the typing function $\tau$ is clear from context, we may abuse notation and denote $(A,\tau)$ simply by $A$.

A morphism $q\colon(A,\tau)\to (A',\tau')$ in $\TFS$ consists of a function $q\colon A\to A'$ that makes the following diagram of finite sets commute:
\[
\begin{tikzcd}[column sep=small]
A \ar[rr,"q"] \ar[rd,"\tau"']
& {}
& A' \ar[ld,"\tau'"]\\
&\Set
\end{tikzcd}
\]
We refer to the morphisms of $\TFS$ as {\em typed functions}. The category $\TFS$ has a monoidal structure $(A,\tau)\sqcup(A',\tau')$, given by disjoint union of underlying sets and the induced function $A\sqcup A'\to\Set$.
\end{definition}

\begin{definition}Given a finite set $A$ and a function $\tau: A \to \TFS$, we denote by $\ol{(A,\tau)} = \ol A$, the cartesian product 
$$\ol{A}\coloneqq\prod_{a\in A}\tau(a).$$
We call the set $\ol{A}$ the \emph{dependent product} of $A$.\end{definition} Taking dependent products is a functor $\TFS\op\to\Set$, i.e., a morphism $q\colon A\to A'$ induces a function $\ol{q}\colon\ol{A'}\to\ol{A}$. In fact this functor is \emph{strong} in the sense that if $A$ and $B$ are finite sets, then there is an isomorphism,
\begin{align}\label{dia:overline is strong}
\ol{A\sqcup B}\iso\ol{A}\times\ol{B}.
\end{align}

\begin{example}\label{ex:TFS}

We give three examples of typed finite sets and their dependent products. 
\begin{enumerate}
\item If $A=\{1,2,3\}$ and $\tau\colon A\to\TFS$ is given by $\tau(1)=\tau(2)=\tau(3)=\ZZ$, then the dependent product is $\ol{(A,\tau)}\iso\ZZ\times\ZZ\times\ZZ$.
\item Let $\singleton$ be an arbitrary one-element set. Consider the typed finite set $\singleton\To{\tau}\TFS$, sending $\ast$ to the set $\tau(\ast)\iso\{{\texttt a,b,\ldots,z}\}$. Then the dependent product is simply $\ol{(\singleton,\tau)}=\{{\texttt a,b,\ldots,z}\}$.
\item Consider the unique function $!\colon\emptyset\to\TFS$. Its dependent product is $\ol{\emptyset}\iso\singleton$, because the empty product is the singleton set. 
\end{enumerate}
\end{example}


\begin{remark}

In this paper, we choose to speak in terms of operads rather than monoidal categories since the overarching idea of modularity---one thing built from many, many small boxes wired into one large box---is operadic. That said, the formal mathematics will generally be written using the framework of symmetric monoidal categories since doing so avoids the use of subscripts in our notation. We are implicitly referring to a functor 
$$\cat{O}\colon\Cat{SMC}\to\Cat{Oprd}$$ 
from the category of symmetric monoidal categories and lax functors to the category of operads---by which we mean symmetric colored operads---and operad functors. If $(\cat{C},\otimes)$ is a symmetric monoidal category then the operad $\cat{O}_{\cat{C}}$ has the same objects as $\cat{C}$, and a morphism $(X_1,\ldots,X_n)\to Y$ in $\cat{O}_{\cat{C}}$ is defined as a morphism $X_1\otimes\cdots\otimes X_n\to Y$ in $\cat{C}$. See \cite[Example 2.1.3]{Lei}.



Note that monoidal categories are used in \cite{BE} and \cite{BF} to study networks, and some confusion may arise in comparing their work to our own, unless care is taken. See \cite{SSR}. One way to see the difference is that we are focusing on the fact that networks nest, i.e., that multiple dynamical systems can be gathered into a network that is itself a single dynamical system. In this way, our work more closely follows the intention of \cite{Bro} or \cite{Har}. However, neither of these references uses category theory, though the latter mentions it as a plausible approach. 


\end{remark}

\chapter{Mode-dependent networks}\label{sec:MDN}

\begin{definition}\label{def:WD}

A \emph{box} is a pair $X=(\inp{X},\outp{X})$, where $\inp{X},\outp{X}\in\Ob\TFS$ are typed finite sets. Each element of $\inp{X}\sqcup\outp{X}$ will be called a \emph{port}.

If $X$ and $Y$ are boxes, we define a \emph{wiring diagram}, denoted $\varphi\colon X\to Y$, to be a pair of typed functions $\varphi=(\inp{\varphi},\outp{\varphi})$:
\begin{align}\label{dia:wd}
\inp{\varphi}&\colon\inp{X}\too\inp{Y}\sqcup\outp{X}
\\\nonumber
\outp{\varphi}&\colon\outp{Y}\too\outp{X}
\end{align}

We define the composition formula for wiring diagrams $\varphi\colon X\to Y$ and $\psi\colon Y\to Z$ as the dotted arrows below, the indicated compositions in $\Cat{TFS}$: 
\begin{equation}\label{eqn:comp law WD}
\begin{tikzcd}[row sep=30pt,column sep=35 pt]
\inp{X}\ar[d,"\inp{\varphi}"']\ar[r,dashed,"\inp{(\psi\circ\varphi)}"]&\inp{Z}\sqcup\outp{X}\\
\inp{Y}\sqcup\outp{X}\ar[r,"\inp{\psi}\sqcup\outp{X}"']&\inp{Z}\sqcup\outp{Y}\sqcup\outp{X}\ar[u,"\inp{Z}\sqcup\outp{\varphi}\sqcup\outp{X}"']
\end{tikzcd}
\qquad
\begin{tikzcd}[row sep=30pt,column sep=15 pt]
\outp{Z}\ar[rr,dashed,"\outp{(\psi\circ\varphi)}"]\ar[dr,"\outp{\psi}"']&&\outp{X}\\
&\outp{Y}\ar[ur,"\outp{\varphi}"']
\end{tikzcd}
\end{equation}
This defines a category, which we call the category of \emph{wiring diagrams}, and denote $\Cat{WD}$. It has boxes as objects and wiring diagrams as morphisms. It has a symmetric monoidal structure defined by disjoint union, which we denote by $\sqcup\colon\Cat{WD}\times\Cat{WD}\to\Cat{WD}$. We denote the operad underlying $\Cat{WD}$ by $\cat{O}_{\Cat{WD}}$.

\end{definition}

\begin{remark}By convention, we label the input ports of a box $P$ by $\inp{P}_a, \inp{P}_b, ...$ and the output ports by $\outp{P}_a, \outp{P}_b,...$. We also often suppress the typing $\tau$. For a zero-input, two-output box $P=\big((\emptyset,!),(\{\outp{P}_a, \outp{P}_b\},\tau)\big)$, where $\tau(\outp{P}_a)= \tau(\outp{P}_b) = L$, we would write $\inp{P}=\emptyset$ and $\outp{P}=\{\outp{P}_a, \outp{P}_a\}$, and we might draw it in any of the following ways:
$$
	
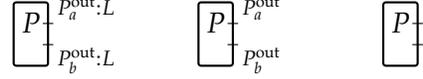
\tikz[wiring diagram,bb port sep=1,bby=8pt,bb min width=13pt,bb port length=2pt,bb rounded corners=1pt,baseline=(B.south)]{
		\node[bb={0}{2},bb name=$P$] (B) {};
		\node[bb={0}{2},right =2cm of B, bb name=$P$] (C) {};
		\node[bb={0}{2},right =2cm of C, bb name=$P$] (D) {};
		\draw[label] node[above right=1pt and 1.2pt of B_out1]     {$\outp{P}_a : L$};
		\draw[label] node[below right=1pt and 1.2pt of B_out2]     {$\outp{P}_b : L$};
		\draw[label] node[above right=1pt and 1.2pt of C_out1]     {$\outp{P}_a$};
		\draw[label] node[below right=1pt and 1.2pt of C_out2]     {$\outp{P}_b$};
	}
$$

\end{remark}

\begin{example}Consider a wiring diagram with three boxes $A_1,A_2,A_3$ wired into a bigger box, $B$. We can define a morphism $\varphi \colon (A_1,A_2,A_3) \too R$ in $\cat{O}_{\Cat{WD}}$ by specifying equations such as $\inp{\varphi}(\inp{A_{1a}}) = \outp{A}_{3a}$, $\outp{\varphi}(\outp{B}_a) = \outp{A_{2a}}$, and so on. (This is equivalent to a morphism $\varphi : X \to B$ in $\Cat{WD}$ where we use the symmetric monoidal product to define $X = A_1 \sqcup A_2 \sqcup A_3$.) Specifying a diagram is equivalent to specifying the two functions $\inp{\varphi}$ and $\outp{\varphi}$.
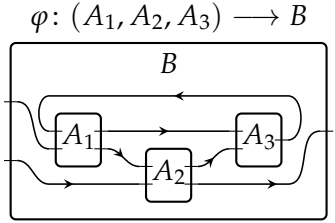
\begin{figure}[h!]
\centering
\[
	\begin{tikzpicture}[wiring diagram,bb port sep=1, bb port length=2.5pt, bbx=.6cm, bb min width=.6cm, bby=1.3ex]
		\node[bb={2}{2}, bb name=${A_1}$] (A1) {};
		\node[bb={2}{2}, below right= -1 and 1 of A1, bb name=${A_2}$] (A2) {};
		\node[bb={2}{1}, above right = -1 and 1 of A2, bb name=${A_3}$] (A3) {};
		\node[bb={2}{1}, fit={(A1) (A2) (A3) ($(A1.north)+(0,3)$)},bb name =${B}$] (B) {};
		\draw(B_in1') to (A1_in2);
		\draw[ar] (A1_out1) to (A3_in1);
		\draw[ar] (A1_out2) to (A2_in1);
		\draw[ar] (A2_out1) to (A3_in2);
		\draw[ar] let \p1=(A2.south west), \p2=(A1.south west), \n1={\y1+\bby}, \n2=\bbportlen in
	    (B_in2') to (\x2-\n2,\n1) -- (\x2+\n2,\n1) to (A2_in2);
		\draw[ar] let \p1=(A2.south east), \p2=(A3.south east), \n1={\y1+\bby}, \n2=\bbportlen in
	    (A2_out2) to (\x2-\n2,\n1) -- (\x2+\n2,\n1) to (B_out1');
	    	\draw[ar] let \p1=(A3.north east), \p2=(A1.north west), \n1={\y1+\bby}, \n2=\bbportlen in
	    (A3_out1) to[in=0] (\x1+\n2,\n1) -- (\x2-\n2,\n1) to[out=180] (A1_in1);
		\node[anchor=south] at (B.north) {$\varphi\colon (A_1,A_2,A_3)\too B$};
	\end{tikzpicture}
\]
\caption{A wiring diagram.}
\label{fig:wdexample}
\end{figure}
Note that the connections in a wiring diagram are all $n$-to-one for some $n$, e.g. wires may naturally ``split'' but may not combine unless through a box. Every output node maps to at most one input node, but there may be many output nodes for one input node. See also \cite[Definition 3.5]{VSL}.
\end{example}

By now we can define static networks, e.g. pictures. To define networks that can morph and change, we give the following definition:


\begin{definition}\label{def:MDN}

We define a symmetric monoidal category, called the category of \emph{mode-dependent networks} and denoted $\Cat{MDN}$, as follows. An object in $\Cat{MDN}$, called a \emph{modal box}, is a pair $(M,X)$, where $M\in\Set$ is a set and $X\colon M\to\Ob\Cat{WD}$ is function, sending each mode to a box (see Definition~\ref{def:WD}). We call $M$ the \emph{set of communicative modes}, or just \emph{mode set} for short, and we call $X$ the \emph{interface function}. If $X$ factors through a one-element set $\singleton$, we call $(M,X)$ a \emph{mode-independent box}. 

We define  $\Hom_{\Cat{MDN}}((M,X),(N,Y))$ to be the set of pairs $(\epsilon,\sigma)$, where $\epsilon : M \to \Mor \Cat{WD}$ and $\sigma : M \to N$ are functions making the following diagram commute:
\begin{equation}\label{eqn:MDN}
	\begin{tikzcd}
		M\ar[d,"X"']\ar[r,equal]&M\ar[d,"\epsilon"]\ar[r,"\sigma"]&N\ar[d,"Y"]\\
		\Ob\Cat{WD}&\Mor\Cat{WD}\ar[r,"\cod"']\ar[l,"\dom"]&\Ob\Cat{WD}
	\end{tikzcd}
\end{equation}
We call a morphism $(\epsilon,\sigma)\in\Hom_{\Cat{MDN}}((M,X),(N,Y))$ a \emph{mode-dependent network}, and we call $\epsilon$ the \emph{event map}, following \cite{field}. We can also write
$$
\epsilon\in\prod_{m\in M}\Cat{WD}(X(m),Y(m)).
$$
In the special case that $X$, $Y$, and $\epsilon$ factor through the one-element set $\singleton$, i.e., if neither the shape of the boxes nor the wiring diagram change with the mode, then we say the network is \emph{mode-independent}.

Given two composable morphisms (mode-dependent networks) 
$$
	\begin{tikzcd}[column sep = 1.75em]
			M_0\ar[d,"X_0"']\ar[r,equal]&M_0\ar[d,"\epsilon_0"]\ar[r,"\sigma_0"]&M_1\ar[d,"X_1"]\ar[r,equal]&M_1\ar[d,"\epsilon_1"]\ar[r,"\sigma_1"]&M_2\ar[d,"X_2"]\\
			\Ob\Cat{WD}&\Mor\Cat{WD}\ar[r,"\cod"']\ar[l,"\dom"]&\Ob\Cat{WD}&\Mor\Cat{WD}\ar[l,"\dom"]\ar[r,"\cod"']&\Ob\Cat{WD}
	\end{tikzcd}
$$
the composition formula is given by $(\epsilon_1,\sigma_1)\circ_{\Cat{MDN}}(\epsilon_0,\sigma_0)\coloneqq(\epsilon,\sigma)$, where $\sigma\coloneqq\sigma_1\circ\sigma_0$ is the composition of functions, and where $\epsilon\colon M_0\to\Mor\Cat{WD}$ is given by composition in $\Cat{WD}$: 
\begin{align}\label{dia:comp of events}
\epsilon(m_0)\coloneqq \epsilon_1\sigma_0(m_0)\circ \epsilon_0(m_0).
\end{align}

The monoidal structure on $\Cat{MDN}$ is given on objects by 
\begin{equation}\label{eq:monoidal MDN}
(M,X)\otimes (N,Y)\coloneqq(M\times N,X\sqcup Y),
\end{equation}
and similarly on morphisms. Here, $X\sqcup Y$ is shorthand for the composite
$$M\times N\To{X\times Y}\Ob\Cat{WD}\times\Ob\Cat{WD}\To{\sqcup}\Ob\Cat{WD}.$$ 
The unit object in $\Cat{MDN}$ is $(\singleton,\emptyset)$, where $\singleton$ is the singleton set, and $\emptyset$ is the monoidal unit of $\Cat{WD}$.

\end{definition}

\begin{remark}

In \cite{field}, attention is paid to the image of the event map $\epsilon\colon M\to\Cat{WD}(X,Y)$, which is denoted $\cat{A}\ss\Cat{WD}(X,Y)$. At certain points in that discussion, $\cat{A}$ is chosen independently of $\epsilon$, but at others it is assumed to be the image of $\epsilon$, see \cite[Remark 4.10(2)]{field}. In Definition~\ref{def:MDN}, we could have defined a morphism in $\Cat{MDN}$ to consist of a triple $(\cat{A},\epsilon,\sigma)$, where $\cat{A}\ss\Cat{WD}(X,Y)$ and $\epsilon\colon M\to\cat{A}$. In this case, composition would also involve composing the various $\cat{A}$'s, but this is straightforward. However, specifying $\cat{A}$ independently seemed superfluous here, especially given Field's remark.

\end{remark}

\begin{example}[Building a basic model]\label{ex:MDN} 
The human retina is made up of light-sensitive nerve cells which convert any image projected onto the retinal detectors (rods and cones) into neural signals. Like almost all neurons, retinal nerve cells have three physiological modes: polarized (the neuron is inactive), depolarized (the neuron is active and firing), and hyperpolarized (the neuron has just fired and is not receiving input). In the retina the situation is somewhat more complicated; there is an extra mechanism that tracks the cell's adaptation to background luminance \cite{Jarsky}. This is modeled by the two wiring diagrams, $\phi$ and $\psi$, in Figure~\ref{fig:retinal nerve}. 


\begin{figure}[h!]
\centering
\[
	\begin{tikzpicture}[wiring diagram,bb port sep=1, bb port length=2.5pt, bbx=.6cm, bb min width=.6cm, bby=1.5ex]
		\node[bb={1}{1}] (N) {$N$};
		\node[bb={1}{1}, fit={(N) ($(N.south)+(3,0)$)}, bb name=$R$] (box) {};
		\draw[ar] (box_in1') to (N_in1);
		\draw[ar] (N_out1) to (box_out1');
		\node[bb={0}{1}, right= 10 of N, bb name=$N$] (N2) {};
		\node[bb={1}{1}, fit={(N2) ($(N2.south)+(3,0)$)}, bb name=$R$] (box2) {};
		\draw[ar] (N2_out1) to (box2_out1');
		\draw[label] 
		    node[below = 3pt of box.south]     {$m_N = \text{polarized}$};
		\draw[label]
		    node[below = 3pt of box2.south]     {$m_N = \text{depolarized, hyperpolarized}$};
		\draw[label]
		    node[above = 3pt of box.north]     {$\varphi : N(m_N) \to N(\ast)$};
		\draw[label]
		    node[above = 3pt of box2.north]     {$\psi : N(m_N) \to N(\ast)$};
		\draw[label] node[left=1.2pt of box_in1]     {$\inp{N}_a : [0,1]$};
		\draw[label] node[left=1.2pt of box2_in1]     {$\inp{N}_a : [0,1]$};
		\draw[label] node[right=1.2pt of box_out1]     {$\outp{N}_a : \{0,1\}$};
		\draw[label] node[right=1.2pt of box2_out1]     {$\outp{N}_a : \{0,1\}$};
	\end{tikzpicture}
\]
\caption{Three modes of a retinal nerve, represented by the wiring diagrams $\varphi, \psi$.}
\label{fig:retinal nerve}
\end{figure}
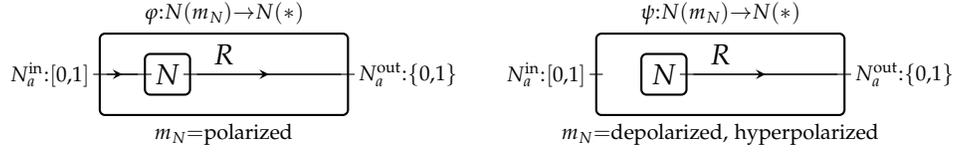

The modal box $(M_N, N)$ represents the nucleus of a retinal nerve cell which takes a value in $[0,1]$ as input and outputs a spike in $\{0,1\}$. The outer box $(M_R, R)$ is a simply a container for $(M_N,N)$. The different modes $M_N = \{ \text{polarized, depolarized, hyperpolarized} \}$ each specify a (normal) box in $\Ob \Cat{WD}$:
\begin{align*}
N(\text{depolarized}) &= \begin{tikzpicture}[wiring diagram,bb small]\node[bb={1}{1}]{};\end{tikzpicture} = (\inp{N}, \outp{N}) = ( \{\inp{N}_a, \inp{N}_b \}, \{ \outp{N}_a, \outp{N}_b \} ) \\
N(\text{hyperpolarized}) &= \begin{tikzpicture}[wiring diagram,bb small]\node[bb={1}{1}]{};\end{tikzpicture} = (\inp{N}, \outp{N}) = ( \{\inp{N}_a, \inp{N}_b \}, \{ \outp{N}_a, \outp{N}_b \} ) \\ 
N(\text{polarized}) &= \begin{tikzpicture}[wiring diagram,bb small]\node[bb={0}{1}]{};\end{tikzpicture} = (\inp N, \outp N) = ( \{ \inp N_b \}, \{ \outp N_a, \outp N_b \} ) \\
\end{align*}
The box $(M_R,R) = R = \begin{tikzpicture}[wiring diagram,bb small]\node[bb={1}{1}]{};\end{tikzpicture}$ is mode-independent, so we represent its mode set by $M_R = \{\ast\}$. The entire mode-dependent network in Figure~\ref{fig:retinal nerve} is defined by a morphism $\text{Nerve} = (\epsilon, \sigma_!) : (M_N, N) \to (M_R, R)$ in $\Cat{MDN}$ with the following data:
\begin{align*}
\epsilon(\text{depolarized}) &= \epsilon(\text{hyperpolarized}) = \varphi \in \Mor\Cat{WD} \\ 
\text{where} \quad &\inp{\varphi}(\inp{N}_a) = \inp{R}_a, \\
&\inp{\varphi}(\inp{N}_b) = \outp{N}_b, \\
&\outp{\varphi}(\outp{R}_a) = \outp{N}_a \\
\epsilon(\text{polarized}) &= \psi \in \Mor\Cat{WD} \\
\text{where} \quad &\inp{\psi}(\inp{N}_b) = \outp{N}_b \\
& \outp{\psi}(\outp{R}_a) = \outp{N}_a \\
\sigma_!(\text{polarized}) &= \sigma_!(\text{depolarized}) = \sigma_!(\text{hyperpolarized}) = \ast 
\in M_R
\end{align*}

One then checks that the data satisfies the commutative diagram in the definition of a mode-dependent network, e.g. Definition~\ref{def:MDN}.
\end{example}

\begin{example}[Forming products and compositions]\label{ex:products}
We can form a simple model of the eye by taking the (symmetric, monoidal) product of all the nerve cells, i.e. $E = R_1 \sqcup R_2 \sqcup ... \sqcup R_n$. (Note, $(M_E, E)$ has mode set $M_E = \prod M_{R_i} \simeq \{\ast\}$.) To define a mode-dependent network on $E$ we then specify a morphism \[ \text{Eye} = (\epsilon, \sigma_!) : ((M_{R_1}, R_1), ..., (M_{R_n}, R_n)) \to (M_E, E).\] This morphism lives in the operad $\cat{O}_\Cat{MDN}$, so composing it with another morphism means ``zooming in'' on the diagram, while precomposing means ``zooming out''. For example, to display the interiors of each $R_i$, we precompose the mode-dependent network $\text{Eye} : ((M_{R_1}, R_1), ..., (M_{R_3}, R_3)) \to (M_E, E)$ with three copies of $\text{Nerve} : (M_N, N) \to (M_R, R)$, as in Figure~\ref{fig:eyeball}.

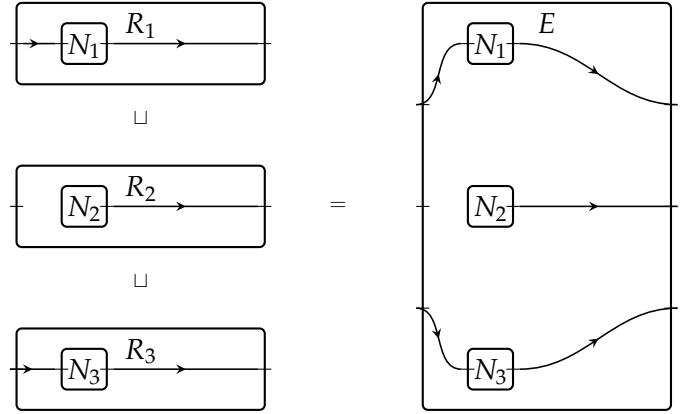
\begin{figure}[h!]
\centering
\[
	\begin{tikzpicture}[wiring diagram,bb port sep=1, bb port length=2.5pt, bbx=.6cm, bb min width=.6cm, bby=1.5ex]
		\node[bb={1}{1}, bb name=$N_1$] (R) {};
		\node[bb={1}{1}, fit={(R) ($(R.south)+(3,1)$)}, bb name=$R_1$] (box) {};
		\draw[ar] (box_in1') to (R_in1);
		\draw[ar] (R_out1) to (box_out1');
		\node[bb={0}{1}, below= 6 of R, bb name=$N_2$] (RK) {};
		\node[bb={1}{1}, fit={(RK) ($(RK.south)+(3,1)$)}, bb name=$R_2$] (box2) {};
		\draw[ar] (RK_out1) to (box2_out1');

		\node[bb={1}{1}, below= 6 of RK, bb name=$N_3$] (RL) {};
		\node[bb={1}{1}, fit={(RL) ($(RL.south)+(3,1)$)}, bb name=$R_3$] (box3) {};
		\draw[ar] (box3_in1) to (RL_in1);
		\draw[ar] (RL_out1) to (box3_out1');
	
		\node[bb={1}{1}, right=8 of R, bb name=${N_1}$] (R1) {};
		\node[bb={0}{1}, right=8 of RK, bb name=${N_2}$] (R2) {};
		\node[bb={1}{1}, right=8 of RL, bb name=${N_3}$] (R3) {};
		\node[bb={3}{3}, fit={(R1) (R2) (R3) ($(R1.north)+(3,0)$)},bb name =${E}$] (E) {};
		\draw[ar] (E_in1) to (R1_in1);
		\draw[ar] (E_in3) to (R3_in1);
		\draw[ar] (R1_out1) to (E_out1);
		\draw[ar] (R2_out1) to (E_out2);
		\draw[ar] (R3_out1) to (E_out3);
		\draw[label]
		    node[below = 10pt of box.south]     {$\sqcup$};
		\draw[label]
		    node[below = 10pt of box2.south]     {$\sqcup$};
		\draw[label]
		    node[right = 24pt of box2]     {$=$};
	\end{tikzpicture}
\]
\caption{The monoidal product of multiple mode-dependent networks, displayed with an instance of the mode-dependent network $\text{Eye} \circ (\text{Nerve}, \text{Nerve}, \text{Nerve})$.}
\label{fig:eyeball}
\end{figure}


\end{example}

\begin{proposition}

Every wiring diagram is canonically a mode-independent network in the following sense: There is a strong monoidal functor 
$$I\colon\Cat{WD}\to\Cat{MDN}$$
given by $I(X)\coloneqq(\singleton,X)$, where $\singleton$ is any choice of singleton set.

\end{proposition}

\begin{proof}

This is straightforward. A wiring diagram $\varphi\in\Cat{WD}(X,Y)$ is sent to $I(\varphi)\coloneqq(\epsilon_\varphi,\singleton)\in\Cat{MDN}(I(X),I(Y))$, where $\epsilon_\varphi\colon\singleton\to\Cat{WD}(X,Y)$ picks out $\varphi$, and $\singleton$ is the identity function on the singleton set. This network is mode-independent because it has only one mode. It is easy to check that $I$ is a strong monoidal functor.
\end{proof}

Given any mode-dependent network $(\epsilon,\sigma)\colon(M,X)\to(N,Y)$ in $\Cat{MDN}$, it is often convenient to factor $\epsilon\in\prod_{m\in M}(X(m),Y(m))$ into two parts,
\begin{align*}
\inp{\epsilon}&\in \prod_{m\in M}\TFS\big(\inp{X}(m),\inp{Y}(m)\sqcup\outp{X}(m)\big)\\
\outp{\epsilon}&\in\prod_{m\in M}\TFS\big(\outp{Y}(m),\outp{X}(m)\big).
\end{align*}
These are obtained simply by considering the two components in (\ref{dia:wd}). Composing with the dependent product functor $\Cat{TFS}\op\to\Set$ from Definition~\ref{def:typed finite sets}, we obtain the following dependent functions:
\begin{align*}
\vinp{\epsilon}&\in\prod_{m\in M}\Set\big(\vinp{Y}(m)\times\voutp{X}(m),\vinp{X}(m)\big)\\
\voutp{\epsilon}&\in\prod_{m\in M}\Set\big(\voutp{X}(m),\voutp{Y}(m)\big).
\end{align*}
For any communicative mode $m\in M$, these functions specify how information will travel within the network ($\vinp{\epsilon}$) and how information will be exported from the network ($\voutp{\epsilon}$).

\chapter{Dynamical systems on mode-dependent networks}\label{sec:DS_MDN}
The goal of this paper is to define an algebra, i.e., a lax functor $\fun{P}\colon\Cat{MDN}\to\Cat{Set}$, of (synchronous) discrete dynamical systems on a mode-dependent network. Below, we will supply the data that defines $\fun{P}$ in Definition~\ref{def:P} and then prove that it satisfies the conditions of being an algebra in Proposition~\ref{prop:main result}.

Let $(M,X)\in\Ob\Cat{MDN}$ be a modal box, where $M\in\Set$ is a mode set and $X=(\inp{X},\outp{X})\colon M\to\Ob\Cat{WD}$ is a modal box. Here $\inp{X},\outp{X}\colon M\to\TFS$. Recall from Definiton~\ref{def:typed finite sets} the notation $\vinp{X},\voutp{X}\colon M\to\Set$ for their dependent products. We define the set $\fun{P}(M,X)$ as
\begin{align}\label{eq:P on object}
\fun{P}&(M,X)\coloneqq
\left\{ 
		(S,q,\inp{f},\outp{f}) \;\;\middle|\;\;
	\parbox{2.6in}{\centering\linespread{1.5}
		$S\in\Set$,\quad$q\colon S\to M,$\\
		$\inp{f}\in\displaystyle\prod_{s\in S}\Set\Big(\vinp{X}\big(q(s)\big),S\Big),$\\
		$\outp{f}\in\displaystyle\prod_{s\in S}\voutp{X}\big(q(s)\big)$
	}\right
\}
\end{align}
That is, an element of $\fun{P}(M,X)$ is a 4-tuple consisting of:
\begin{itemize}
\item a set $S\in\Set$, called the \emph{state set};
\item a function $q\colon S\to M$, called the \emph{underlying mode function};
\item for each $s\in S$, with underlying mode $m=q(s)$, a function $\inp{f}(s)\colon\vinp{X}(m)\to S$, called the \emph{state update function}; and
\item for each $s\in S$, with underlying mode $m=q(s)$, an element $\outp{f}\in\voutp{X}(m)$, called the \emph{readout}.
\end{itemize}
We may denote an element of $\fun{P}(M,X)$ simply by $(S,q,f)$, and we call $f=(\inp{f},\outp{f})$ an \emph{open dynamical system} with state set $S$, following \cite{VSL}. The whole 3-tuple $(S,q,f)$ will be called a \emph{modal dynamical system}.

For the lax monoidal structure, one coherence map, $\fun{P}(M,X)\times\fun{P}(N,Y)\to\fun{P}(M\times N,X\sqcup Y)$, is given by cartesian products
\begin{align}
\nonumber
S_{X\times Y}&\coloneqq S_X\times S_Y
\\\label{dia:coherence for P}
q_{X\times Y}&\coloneqq q_X\times q_Y
\\\nonumber
f_{X\times Y}&\coloneqq f_X\times f_Y 
\end{align}
Note the isomorphism (\ref{dia:overline is strong}). The other coherence map, $\singleton\to\fun{P}(\singleton,\emptyset)$, is the element $(\singleton,!,!,!)$, where each $!$ denotes the unique morphism of the evident type.

Given a morphism $(\epsilon,\sigma)\colon (M,X)\to (N,Y)$ in $\Cat{MDN}$, we need to provide a function $\fun{P}(\epsilon,\sigma)\colon\fun{P}(M,X)\to\fun{P}(N,Y)$. For an arbitrary modal dynamic system $(S,q,f)\in\fun{P}(M,X)$, we define 
\begin{equation}\label{eq:P on morphisms}
\fun{P}(\epsilon,\sigma)(S,q,f)\coloneqq(S,r,g)
\end{equation}
where $S$ is unchanged, $r=\sigma\circ q$ is the composite $S\To{q}M\To{\sigma}N$, and for every $s\in S$ the update function $\inp{g}(s)$ and the readout $\outp{g}(s)$ are given as follows. Let $m=q(s)$, so we have $\vinp{\epsilon}(m)\colon\vinp{Y}(m)\times\voutp{X}(m)\to\vinp{X}(m)$ and $\voutp{\epsilon}(m)\colon\voutp{X}(m)\to\voutp{Y}(m)$. Then the readout is given by
\begin{align}\label{dia:P on morphisms--out}
	\outp{g}(s)=\voutp{\epsilon}(m)(\outp{f}(s))
\end{align}
and the update $\inp{g}(s)$ is given by the following composition in $\Set$:
\begin{align*}
	\begin{tikzcd}[column sep=7em, row sep = 7ex, ampersand replacement=\&]
		\vinp{Y}(m)\ar[d,dashed,"\inp{g}(s)"']\ar[r,"{\vinp{Y}(m)\times\outp{f}(s)}"]\&\vinp{Y}(m)\times\voutp{X}(m)\ar[d,"\vinp{\epsilon}(m)"]\\
		S\&\vinp{X}(m)\ar[l,"\inp{f}(s)"]
	\end{tikzcd}
\end{align*}
We can restate this (still with $m=q(s)$) as follows:
\begin{equation}\label{eqn:equational restatement}
\inp{g}(s)(y)=\inp{f}(s)\Big(\vinp{\epsilon}(m)\big(y,\outp{f}(s)\big)\Big)
\end{equation}

\begin{definition}\label{def:P}

We define $\fun{P}$ as the data (\ref{eq:P on object}), (\ref{dia:coherence for P}), and (\ref{eq:P on morphisms}) above. So far they are only data; we will show that they constitute a lax monoidal functor $\fun{P}\colon\Cat{MDN}\to\Set$ in Proposition~\ref{prop:main result}. 

\end{definition}

\begin{remark}

If $(M,X)$ is a mode-independent box, we can write $X=X(m)$ for all $m\in M$ at which point we can rewrite the state update and readout as functions $\inp{f}\colon S\times \vinp{X}\to S$ and $\outp{f}\colon S\to\voutp{X}$. 

Here is the intuitive idea of how information flows through a mode-dependent network. It begins with each inner box $X_i$ converting its state to an output. These outputs, together with the input to $Y$, are carried through the wires and either exported from $Y$ or fed as inputs to the $X_i$. These inner boxes then use this input to update their state, and hence their communicative mode, and hence their own shape and that of the network. The process repeats.
\end{remark}

\begin{example}\label{ex:dynamics}
Before proving the main result, we motivate it by defining a modal dynamical system $(S_N, q_N, f_N) \in \mathcal{P}(M_N, N)$ on our simple model of the retinal nerve $(M_N, N)$ in Example~\ref{ex:MDN}. For convenience, we reproduce the diagram from Figure~\ref{fig:retinal nerve}.

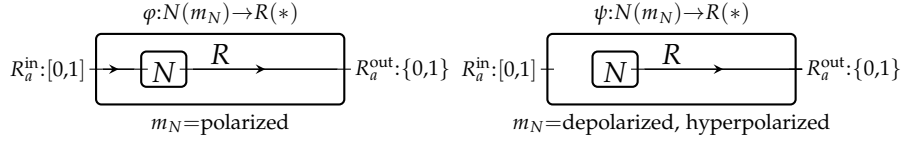
\[
	\begin{tikzpicture}[wiring diagram,bb port sep=1, bb port length=2.5pt, bbx=.6cm, bb min width=.6cm, bby=1.3ex]
		\node[bb={1}{1}, bb name=$N$] (R) {};
		\node[bb={1}{1}, fit={(R) ($(R.south)+(3,0)$)}, bb name=$R$] (box) {};
		\draw[ar] (box_in1') to (R_in1);
		\draw[ar] (R_out1) to (box_out1');
		\node[bb={0}{1}, right= 9 of R, bb name=$N$] (R2) {};
		\node[bb={1}{1}, fit={(R2) ($(R2.south)+(3,0)$)}, bb name=$R$] (box2) {};
		\draw[ar] (R2_out1) to (box2_out1);
		\draw[label] 
		    node[below = 3pt of box.south]     {$m_N = \text{polarized}$};
		\draw[label]
		    node[below = 3pt of box2.south]     {$m_N = \text{depolarized, hyperpolarized}$};
		\draw[label]
		    node[above = 3pt of box.north]     {$\varphi : N(m_N) \to R(\ast)$};
		\draw[label]
		    node[above = 3pt of box2.north]     {$\psi : N(m_N) \to R(\ast)$};
		\draw[label] node[left=1.2pt of box_in1]     {$\inp{R}_a : [0,1]$};
		\draw[label] node[left=1.2pt of box2_in1]     {$\inp{R}_a : [0,1]$};
		\draw[label] node[right=1.2pt of box_out1]     {$\outp{R}_a : \{0,1\}$};
		\draw[label] node[right=1.2pt of box2_out1]     {$\outp{R}_a : \{0,1\}$};
	\end{tikzpicture}
\]

\begin{description}
\item[State sets.] Let $S_N = M_N \times A$, where $A = \mathbb{R}^+$ is the adaptation of the cell to light---the more adapted to light the cell is, the more light needed to activate the cell. The wrapper $R$ has state set equal to the product of its internal boxes, in this case just $N$, so $S_R = S_N$. $M_R = \{\ast\}$.
\item[Underlying mode function.] We define $q_N : S_N \to M_N$ to be the first projection. $q_R: S_R \to M_R$ is trivial, since $M_R = \{\ast\}$.
\item[State update function.] Define two constants, a firing threshold $\alpha \in [0,1]$ and an adaptation constant $\beta > 1$. For $m = \text{polarized}$, let $\inp{f}_N : S_N \times [0,1] \to S_N$ be defined by 
\[
\inp{f}_N(m, a, x) = 
\begin{cases}
(\text{depolarized}, a + x/\beta) & x - a \geq \alpha \\ 
(\text{polarized}, a/\beta) & \text{otherwise}
\end{cases}
\]
In other words, the cell depolarizes (``fires'') from its inactive state if the input is above a certain threshold $\alpha$, and simultaneously increases its stored adaptation parameter $a \mapsto a + x/\beta$.

If the cell is either its depolarized or hyperpolarized modes, then there is no separate input for light received, so the update function $\inp{f}_N$ depends only on the previous state, e.g. \[\inp{f}_N(m,a) = \begin{cases}
(\text{hyperpolarized}, a/\beta) & m = \text{depolarized} \\
(\text{polarized}, a/\beta) & m = \text{hyperpolarized}
\end{cases}.
\]
The state update function for the wrapper $(M_R,R)$ is just the same as that for $(M_N, N)$.
\item [Readout.] In this case we can define the same readout function $\outp{f}_N : S_N \to \{0,1\}$ for all modes: \[\outp{f}_N(m,a) = \begin{cases}1 & m = \text{polarized} \\ 0 & \text{otherwise}\end{cases}.\]
$\outp{f}_R$ is defined identically.
\end{description}

The point of calling $\mathcal{P} : \Cat{MDN} \to \Set$ an algebra (i.e. a functor) is that one can construct new modal dynamical systems using essentially algebraic operations. For example, let 
\[\text{Eye'} = \text{Eye} \circ (\text{Nerve}, ..., \text{Nerve})\] 
be the mode-dependent network of the ``eyeball'' we showed in Figure~\ref{fig:eyeball}. Then we can obtain a new modal dynamical system $(S_E, q_E, f_E) \in \mathcal{P}(M_E, E)$ by applying a function 
\[\mathcal{P}(\text{Eye'}) : \mathcal{P}((M_{N_1}, N_1) \otimes (M_{N_2}, N_2) \otimes ...)) \to \mathcal{P}(M_E,E)\]
to $(S_N, q_N, f_N)$ as defined above. Specifying 
\[(S_E, q_E, f_E) = \mathcal{P}(\text{Eye'})((S_{N_1}, q_{N_1}, f_{N_1}), ..., (S_{N_k}, q_{N_k}, f_{N_k}))\] is fairly straightforward: we first obtain an intermediate modal dynamical system on $(M_E,E)$---think of this as the most `general' dynamical system we can place on $(M_E,E)$---by using the given coherence map (\ref{dia:coherence for P}) to define a dynamical system with the following data
\begin{align*}
S_E &= S_{N_1} \times S_{N_2} \times ... \times S_{N_k} \\
q_E &= q_{N_1} \times  q_{N_2} \times ... \times q_{N_k} \\
\inp{f}_E &= \inp{f}_{N_1} \times \inp{f}_{N_2} \times ... \times \inp{f}_{N_k}\\
\outp{f}_E &= \outp{f}_{N_1} \times \outp{f}_{N_2} \times ... \times \outp{f}_{N_k}
\end{align*}
We can then define a new dynamical system by defining a map on this dynamical system according to Equations~\eqref{dia:P on morphisms--out}~and~\eqref{eqn:equational restatement}.

Since $\text{Eye'}$ here is just the disjoint union of the several retinal nerves (wrapped up in wrappers $R_i$), this intermediate modal dynamical system is just what we want.

\end{example}


\begin{proposition}\label{prop:main result}

With the data given in Definition~\ref{def:P}, the map $\fun{P}\colon\Cat{MDN}\to\Set$ is a lax monoidal functor.

\end{proposition}

\begin{proof}

It is clear that the coherence maps satisfy the necessary unitality and associativity properties, because they are just given by cartesian products. So it remains to show that $\fun{P}$ is functorial, i.e., that it commutes with composition. 

Suppose given morphisms
$$(M_0,X_0)\To{(\epsilon_0,\sigma_0)}(M_1,X_1)\To{(\epsilon_1,\sigma_1)}(M_2,X_2)$$ 
in $\Cat{MDN}$, and let $(\epsilon,\sigma_1\circ\sigma_0)=(\epsilon_1,\sigma_1)\circ(\epsilon_0,\sigma_0)$ be the composite as in (\ref{dia:comp of events}). Suppose that $(S,q_0,f_0)\in\fun{P}(M_0,X_0)$ is an arbitrary element. For notational convenience, define
\begin{align*}
(S,q_1,f_1)&\coloneqq\fun{P}(\epsilon_0,\sigma_0)(S,q_0,f_0),\\
(S,q_2,f_2)&\coloneqq\fun{P}(\epsilon_1,\sigma_1)(S,q_1,f_1),\\ 
(S,q_{2'},f_{2'})&\coloneqq\fun{P}(\epsilon,\sigma_1\circ\sigma_0)(S,q_0,f_0),
\end{align*}
see (\ref{eq:P on morphisms}). We want to show that $(S,q_{2},f_{2})=^?(S,q_{2'},f_{2'})$. It is obvious that
$$q_{2}=q_{2'}=\sigma_1\circ\sigma_0\circ q_0,$$ 
so it only remains to show that $f_{2}=^?f_{2'}$, i.e., that $\inp{f_{2}}=^?\inp{f_{2'}}$ and $\outp{f_{2}}=^?\outp{f_{2'}}$. 

We begin by choosing any $s\in S$, and we let $m_0=q_0(s)$, $m_1=q_1(s)=\sigma_0(q_0(s))$, and $m_2=q_2(s)$. For the first desired equation, we are trying to prove the equality $\inp{f_{2}}(s)=^?\inp{f_{2'}}(s)$ between functions $\vinp{X_2}(m_2)\to S$. Thus we choose an element $x_2\in\vinp{X_2}(m_2)$ and, using \eqref{dia:P on morphisms--out} and \eqref{eqn:equational restatement}, we have:
\begin{align*}
\inp{f_2}(s)(x_2)
&=\inp{f_1}(s)\Big(\vinp{\epsilon_1}(m_1)\big(x_2,\outp{f_1}(s)\big)\Big)
\\
&=\inp{f_0}(s)\bigg(\vinp{\epsilon_0}(m_0)\Big(\vinp{\epsilon_1}(m _1)\big(x_2,\outp{f_1}(s)\big),\outp{f_0}(s)\Big)\bigg)\\
&=\inp{f_0}(s)\Bigg(\vinp{\epsilon_0}(m_0)\bigg(\vinp{\epsilon_1}(m_1)\Big(x_2,\voutp{\epsilon_0}(m_0)\big(\outp{f_0}(s)\big)\Big),\outp{f_0}(s)\bigg)\Bigg)\\
&=\inp{f_0}(s)\Big(\vinp{\epsilon}(m_0)\big(x_2,\outp{f_0}(s)\big)\Big)=\inp{f_{2'}}(s)(x_2)
\end{align*}
For the second desired equation, we are trying to prove the equality $\outp{f_{2}}(s)=^?\outp{f_{2'}}(s)$ between elements of $\voutp{X_2}$. Again using \eqref{dia:P on morphisms--out}, we have:
\begin{align*}
\outp{f_2}(s)
&=
\voutp{\epsilon_1}(m_1)\big(\outp{f_1}(s)\big)\\
&=
\voutp{\epsilon_1}(m_1)\Big(\voutp{\epsilon_0}(m_0)\big(\outp{f_0}(s)\big)\Big)\\ 
&=
\voutp{\epsilon}(m_0)\big(\outp{f_0}(s)\big)=\outp{f_{2'}}(s),
\end{align*}
This concludes the proof.
\end{proof}

\chapter{Applications}\label{sec:applications}
In principle, operads are general enough that they can be applied in any instance where nested graphs, nested networks, and nested graphical models are used. Composition in the operad immediately gives the semantics of networks of networks, along with a consistent and procedural way to plug different networks into each other. Mode-dependency models the additional interaction between dynamics and topology. The formalism is well-suited to heterogenous graphical models of dynamical systems with many functional dependencies between variables \cite{VSL}, but it can also be adapted to more homogeneous models where all ports have the same type and only the connectivity matters (e.g. in studies of node failure \cite{Buld, Reis}).


The remainder of this section is dedicated to illustrating some of these applications by extending our primitive model of the eye. 

\begin{example}[Adding behavior]\label{ex:behavior}
One typical task in constructing a model or simulation of a subject is to add some dynamics or behavior to an existing system. Recall our simple model of the eye from Example~\ref{ex:MDN}. It is clearly incomplete---given time, we might want to model ganglion cells that summarize the behavior of several retinal nerve cells, the special dynamics of detector cells (where ``activating'' a rod or cone actually means inhibiting its spontaneous rate of fire), and the many projections from retinal nerves to other nuclei in the cortex. 

For now, we will do something relatively simple: add the ability to ``blink''. That is, we need to change the mode set of our eyeball $E$ from $\{\ast\}$ to something like $\{\text{open, shut}\}$, corresponding to the two wiring diagrams below
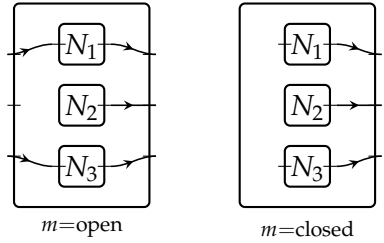
\[
	\begin{tikzpicture}[wiring diagram,bb port sep=1, bb port length=2.5pt, bbx=.6cm, bb min width=.6cm, bby=1.5ex]
		\node[bb={1}{1}, bb name=${N_1}$] (R1) {};
		\node[bb={0}{1}, below=1 of R1, bb name=${N_2}$] (R2) {};
		\node[bb={1}{1}, below=1 of R2, bb name=${N_3}$] (R3) {};
		\node[bb={3}{3}, fit={(R1) (R2) (R3)},bb name =$$] (E) {};
		\draw[ar] (E_in1) to (R1_in1);
		\draw[ar] (E_in3) to (R3_in1);
		\draw[ar] (R1_out1) to (E_out1);
		\draw[ar] (R2_out1) to (E_out2);
		\draw[ar] (R3_out1) to (E_out3);
		\node[bb={1}{1}, right=4 of R1, bb name=${N_1}$] (K1) {};
		\node[bb={0}{1}, below=1 of K1, bb name=${N_2}$] (K2) {};
		\node[bb={1}{1}, below=1 of K2, bb name=${N_3}$] (K3) {};
		\node[bb={0}{3}, fit={(K1) (K2) (K3)},bb name =$$] (K) {};
		\draw[ar] (K1_out1) to (K_out1);
		\draw[ar] (K2_out1) to (K_out2);
		\draw[ar] (K3_out1) to (K_out3);
		\draw[label] node[below=5pt of E]     {$m = \text{open}$};
		\draw[label] node[below=5pt of K]     {$m = \text{closed}$};
	\end{tikzpicture}
\]
It's possible to specify the dynamics for open/shut directly inside $E$; one simply defines the state set of $E$ to be \[ S_E = \{\text{open, shut}\} \times S_{N_1} \times S_{N_2} \times ... \times S_{N_k} \] and the state update function $\inp{f}_E$ to be a similar product \[ \inp{f}_E = \inp{f}_{\text{open/shut}} \times \inp{f}_{N_1} \times \inp{f}_{N_2} \times ... \times \inp{f}_{N_k} \] where $\inp{f}_{\text{open/shut}} : S_E \to \{\text{open/shut}\}$ maps to closed if a sufficient number of neurons are depolarized, and to open otherwise.

For various reasons, this way of specifying the dynamics may not be preferable. For one, it is not biological: the eye offloads this sort of computation to other areas of the brain. Nor is it causal; reflexes do not trigger on stored states but on active signals. Finally, it is not modular: we would like to be able to modify the dynamics of a system directly by rewiring the diagram---i.e. changing the topology---but in this case we would have to change the state update function directly.

As an alternative, we can outsource the computation of $\inp{f}_{\text{open/shut}}$ to another modal box $P$ (here, $P$ stands for pons, an area of the brainstem that mediates the blink reflex). We define $S_P = \{ \text{open, shut} \}$ and $\inp{f}_P : \{0,1\}^n \to S_P$ to map to closed if the average of the signals is above a certain threshold, and to open otherwise. $M_P = \{\ast\}$. See Figure~\ref{fig:blink}. Note that we need to add a new input port to box $E$ to receive an open or shut signal from $P$. 

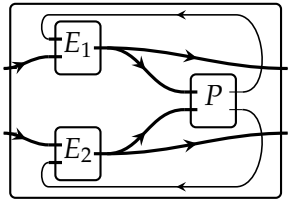
\begin{figure}[h!]
\centering
\[
	\begin{tikzpicture}[wiring diagram,bb port sep=1, bb port length=2.5pt, bbx=.6cm, bb min width=.6cm, bby=1.3ex]
		\node[bbthick={2}{1}{very thick}{very thick}, bb name=${E_1}$] (E1) {};
		\node[bbthick={2}{1}{very thick}{very thick}, below= 3 of E1, bb name=${E_2}$] (E2) {};
		\node[bbthick={2}{2}{very thick}{thin}, below right = 0 and 2 of E1, bb name=${P}$] (C) {};
		\node[bb={2}{2}, fit={(E1) (E2) (C) ($(E1.north)+(0,0)$)},bb name =$$] (V) {};
		\draw[ar, very thick] (E1_out1) to (C_in1);
		\draw[ar, very thick] (E2_out1) to (C_in2);
		\draw[ar] let \p1=(C.north east), \p2=(E1.north west), \n1={\y1+\bby}, \n2=\bbportlen in (C_out1) to[in=0] (\x1,\y2+\n2) -- (\x2-\n2,\y2+\n2) to[out=180] (E1_in1);
		\draw[ar] let \p1=(C.south east), \p2=(E2.south west), \n1={\y1+\bby}, \n2=\bbportlen in (C_out2) to[in=0] (\x1,\y2-\n2) -- (\x2-\n2,\y2-\n2) to[out=180] (E2_in2);
		\draw[ar, very thick] (V_in1) to (E1_in2);
		\draw[ar, very thick] (V_in2) to (E2_in1);
		\draw[ar, very thick] (E1_out1) to (V_out1);
		\draw[ar, very thick] (E2_out1) to (V_out2);
	\end{tikzpicture}
\]
\caption{Two eyes are linked to a modal box $P$ that controls blinking.}
\label{fig:blink}
\end{figure}
\end{example}
For later use, we define this entire system as a mode-dependent network $\text{Blink} : (E_1, E_2, P) \to B$, where $B$ is the containing box.

\begin{example}[Neural networks]\label{ex:nnet}
In (artificial) neural networks, we use neurons to form interconnected layers (there is no connection within a layer). Each neural network can be thought of as a classifier, e.g. a $\{0,1\}$-valued function of the inputs: nodes in the first layer represent features, nodes in the second layer interpret these features into new features, and so on for additional layers. We can model such neural networks as mode-dependent networks and connect them to other mode-dependent networks such as Example~\ref{ex:behavior}. For simplicity, we only consider feedforward networks, e.g. networks of neurons without loops.

A (cortical) \emph{neuron} is a mode-dependent network \[\text{Neuron} = (\epsilon, \sigma_!) : ((M_N, N), (M_S, S)) \to (M_R, R)\] with a nucleus $(M_N, N)$, a summation box $S = (M_S,S)$, and a wrapper $R = (M_R, R)$. Both $S$ and $R$ are mode-independent; $M_N =$ \{polarized, depolarized, hyperpolarized\}, as in Example~\ref{ex:dynamics}. The nucleus is incapable of receiving signals when it is either depolarized (firing), or hyperpolarized (just finished firing).
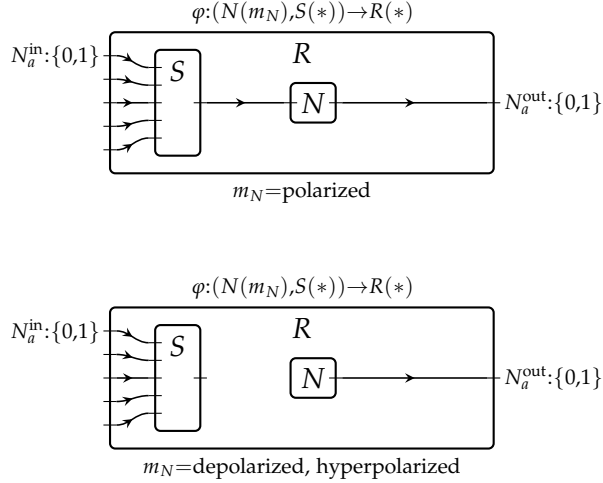
\[
	\begin{tikzpicture}[wiring diagram,bb port sep=1, bb port length=2.5pt, bbx=.6cm, bb min width=.6cm, bby=1.3ex]
		\node[bb={1}{1}] (N) {$N$};
		\node[bb={5}{1}, left=2 of N, bb name=$S$] (S) {};
		\node[bb={5}{1}, fit={(N) (S) ($(N.south)+(3,0)$)}, bb name=$R$] (box) {};
		\draw[ar] (box_in1') to (S_in1);
		\draw[ar] (box_in2') to (S_in2);
		\draw[ar] (box_in3') to (S_in3);
		\draw[ar] (box_in4') to (S_in4);
		\draw[ar] (box_in5') to (S_in5);
		\draw[ar] (S_out1) to (N_in1);
		\draw[ar] (N_out1) to (box_out1');
		\draw[label] 
		    node[below = 3pt of box.south]     {$m_N = \text{polarized}$};
		\draw[label]
		    node[above = 3pt of box.north]     {$\varphi : (N(m_N), S(\ast)) \to R(\ast)$};
		\draw[label] node[left=1.2pt of box_in1]     {$\inp{N}_a : \{0,1\}$};
		\draw[label] node[right=1.2pt of box_out1]     {$\outp{N}_a : \{0,1\}$};
	\end{tikzpicture}
\]

\[
	\begin{tikzpicture}[wiring diagram,bb port sep=1, bb port length=2.5pt, bbx=.6cm, bb min width=.6cm, bby=1.3ex]
		\node[bb={0}{1}] (N) {$N$};
		\node[bb={5}{1}, left=2 of N, bb name=$S$] (S) {};
		\node[bb={5}{1}, fit={(N) (S) ($(N.south)+(3,0)$)}, bb name=$R$] (box) {};
		\draw[ar] (box_in1') to (S_in1);
		\draw[ar] (box_in2') to (S_in2);
		\draw[ar] (box_in3') to (S_in3);
		\draw[ar] (box_in4') to (S_in4);
		\draw[ar] (box_in5') to (S_in5);
		\draw[ar] (N_out1) to (box_out1');
		\draw[label] 
		    node[below = 3pt of box.south]     {$m_N = \text{depolarized, hyperpolarized}$};
		\draw[label]
		    node[above = 3pt of box.north]     {$\varphi : (N(m_N), S(\ast)) \to R(\ast)$};
		\draw[label] node[left=1.2pt of box_in1]     {$\inp{N}_a : \{0,1\}$};
		\draw[label] node[right=1.2pt of box_out1]     {$\outp{N}_a : \{0,1\}$};
	\end{tikzpicture}
\]

To define a modal dynamical system on top of $\text{Neuron} = (\epsilon, \sigma_!)$, we provide the following data: 
\begin{description}
\item[State sets.] Let $k = |\inp{S}|$ be the number of inputs to $S$. Then $S_N = M_N =$ \{polarized, depolarized, hyperpolarized\} and $S_S = \mathbb{R}^+ \times W_S$, where $W_S= \mathbb{R}^k$ denotes a space of weight vectors. As the container, $R$ has state set $S_R = S_N \times S_S$. 
\item[Underlying mode function.] $q_S : S_S \to M_S$ is trivial, since the mode set of $S$ is trivial. Similarly with $q_R$. $q_N : S_N \to M_N$ is just the identity.
\item[State update function.] Define a firing threshold $\alpha \in [0,1]$. The state update function for $\inp{f} : S_N \times \{0,1\}^k \to S_N$ is given by \[\inp{f}_S(s, w_1,...,w_k, x_1, ..., x_k) = (\sum_{i}^k w_i x_i, w_1, ..., w_k)\] where $w \in W_S$. When the mode $N$ is polarized, \[\inp{f}_N(x,m) = \begin{cases}\text{depolarized} & x > \alpha \\ \text{polarized} & \text{otherwise} \end{cases}.\]
When the mode of $N$ is depolarized or hyperpolarized, \[ \inp{f}_N(m) = \begin{cases}\text{hyperpolarized} & m = \text{depolarized} \\ \text{polarized} & m = \text{hyperpolarized}\end{cases}.\]
\item[Readout.] Finally, the readout of $S$ is just $\outp{f}_S(s, w_1, ..., w_k) = s$, while the readout of $N$ is \[ \outp{f}_N(m) = \begin{cases}1 & m = \text{polarized} \\ 0 & \text{otherwise}\end{cases}.\] The readout function for the wrapper $R$ is just the product $\outp{f}_S \times \outp{f}_N$.
\end{description}

Given a modal dynamical system on a neuron, we can form new modal dynamical systems by forming products and compositions of various copies of $\text{neuron}, \sigma_!)$. A \emph{layer} of neurons is a mode-dependent network \[\text{V}_i = ( \epsilon, \sigma_!) : (\text{Neuron} \otimes ... \otimes \text{Neuron}) \to V_i\] formed by a tensor product of neurons, with the condition that there are no connections between neurons inside a layer.
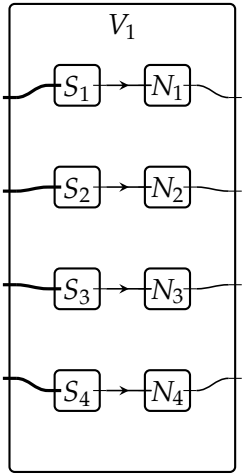
\[
	\begin{tikzpicture}[wiring diagram,bb port sep=1, bb port length=2.5pt, bbx=.6cm, bb min width=.6cm, bby=1.5ex]
		\node[bbthick={1}{1}{very thick}{thin}, bb name=$S_1$] (V11) {};
		\node[bbthick={1}{1}{very thick}{thin}, below= 3 of V11, bb name=$S_2$] (V12) {};
		\node[bbthick={1}{1}{very thick}{thin}, below= 3 of V12, bb name=$S_3$] (V13) {};
		\node[bbthick={1}{1}{very thick}{thin}, below= 3 of V13, bb name=$S_4$] (V14) {};
		\node[bb={1}{1}, right= 1 of V11, bb name=$N_1$] (N1) {};
		\node[bb={1}{1}, right= 1 of V12, bb name=$N_2$] (N2) {};
		\node[bb={1}{1}, right= 1 of V13, bb name=$N_3$] (N3) {};
		\node[bb={1}{1}, right= 1 of V14, bb name=$N_4$] (N4) {};
		
		\node[bbthick={4}{4}{very thick}{thin}, fit={(V11) (V12) (V13) (V14) (N1) (N2) (N3) (N4) ($(V11.north)+(0,2)$) ($(V14.south)+(0,-2)$)}, bb name=${V_1}$] (V1) {};
		\draw[very thick] (V1_in1') to (V11_in1);
		\draw[very thick] (V1_in2') to (V12_in1);
		\draw[very thick] (V1_in3') to (V13_in1);	
		\draw[very thick] (V1_in4') to (V14_in1);
		\draw[ar] (V11_out1) to (N1_in1');
		\draw[ar] (V12_out1) to (N2_in1');
		\draw[ar] (V13_out1) to (N3_in1');
		\draw[ar] (V14_out1) to (N4_in1');		
		\draw (N1_out1) to (V1_out1');
		\draw (N2_out1) to (V1_out2');
		\draw (N3_out1) to (V1_out3');
		\draw (N4_out1) to (V1_out4');		
	\end{tikzpicture}
\]
Just as in Figure~\ref{fig:eyeball}, the diagram above depicts a \emph{composition} of the mode-dependent network $\text{V}_1 : (R_1, ..., R_4) \to V_1$ with several copies of $\text{Neuron} : (S,(M_N, N)) \to R$ to form a new mode-dependent network, $\text{V}_1 \circ (\text{Neuron}, ..., \text{Neuron})$. Again, the composition is what allows us to ``zoom in'' on the structure of each individual neuron. As before, the thicker arrows denote several parallel wires. The dynamics of this system are given directly from the coherence maps in Equation~\ref{dia:coherence for P}.


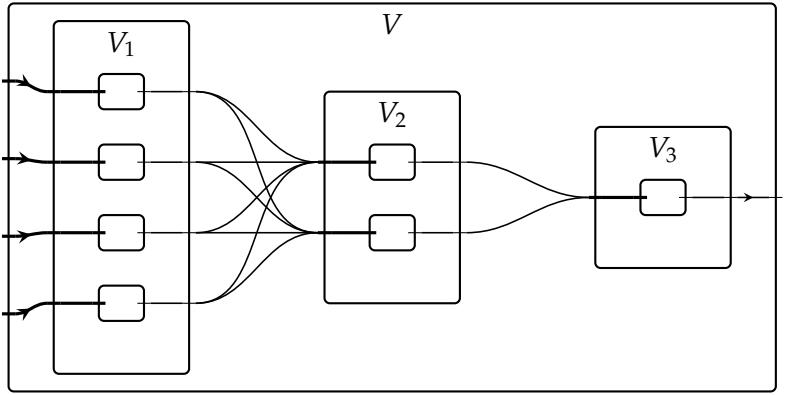
\begin{figure}[h!]
\centering
\[
	\begin{tikzpicture}[wiring diagram,bb port sep=1, bb port length=2.5pt, bbx=.6cm, bb min width=.6cm, bby=1.3ex]
		\node[bbthick={1}{1}{very thick}{thin}, bb name=$$] (V11) {};
		\node[bbthick={1}{1}{very thick}{thin}, below= 2 of V11, bb name=$$] (V12) {};
		\node[bbthick={1}{1}{very thick}{thin}, below= 2 of V12, bb name=$$] (V13) {};
		\node[bbthick={1}{1}{very thick}{thin}, below= 2 of V13, bb name=$$] (V14) {};
		\node[bbthick={4}{4}{very thick}{thin}, fit={(V11) (V12) (V13) (V14) ($(V11.north)+(0,2)$) ($(V14.south)+(0,-2)$)}, bb name=${V_1}$] (V1) {};

		\node[bbthick={1}{1}{very thick}{thin}, right= 5 of V12, bb name=$$] (V21) {};
		\node[bbthick={1}{1}{very thick}{thin}, below= 2 of V21, bb name=$$] (V22) {};
		\node[bbthick={2}{2}{very thick}{thin}, fit={(V21) (V22) ($(V21.north)+(0,2)$) ($(V22.south)+(0,-2)$)}, bb name=${V_2}$] (V2) {};
		\node[bbthick={1}{1}{very thick}{thin}, above right= 0 and 5 of V22, bb name=$$] (V31) {};
		\node[bbthick={1}{1}{very thick}{thin}, fit={(V31) ($(V31.north)+(0,2)$) ($(V31.south)+(0,-2)$)}, bb name=${V_3}$] (V3) {};
		\node[bbthick={4}{1}{very thick}{thin}, fit={(V1) (V2) (V3)}, bb name=${V}$] (V) {};
		\draw[ar, very thick] (V_in1') to (V1_in1);
		\draw[ar, very thick] (V_in2') to (V1_in2);
		\draw[ar, very thick] (V_in3') to (V1_in3);
		\draw[ar, very thick] (V_in4') to (V1_in4);		
		\draw (V1_out1) to (V2_in1);
		\draw (V1_out1) to (V2_in2);
		\draw (V1_out2) to (V2_in1);
		\draw (V1_out2) to (V2_in2);
		\draw (V1_out3) to (V2_in1);
		\draw (V1_out3) to (V2_in2);
		\draw (V1_out4) to (V2_in1);
		\draw (V1_out4) to (V2_in2);
		\draw (V2_out1) to (V3_in1);	
		\draw (V2_out2) to (V3_in1);	
		\draw[ar] (V3_out1) to (V_out1');
		\draw[very thick] (V1_in1') to (V11_in1);
		\draw[very thick] (V1_in2') to (V12_in1);
		\draw[very thick] (V1_in3') to (V13_in1);	
		\draw[very thick] (V1_in4') to (V14_in1);
		\draw (V11_out1) to (V1_out1');
		\draw (V12_out1) to (V1_out2');
		\draw (V13_out1) to (V1_out3');
		\draw (V14_out1) to (V1_out4');
		\draw[very thick] (V2_in1') to (V21_in1);
		\draw[very thick] (V2_in2') to (V22_in1);
		\draw (V21_out1) to (V2_out1');
		\draw (V22_out1) to (V2_out2');
		\draw[very thick] (V3_in1') to (V31_in1);
		\draw (V31_out1) to (V3_out1');

	\end{tikzpicture}
\]
\caption{A feedforward network composed of three layers, written as a mode-dependent network $\text{V} = (\epsilon, \sigma_!) : (V_1, V_2, V_3) \to V$ in $\cat{O}_\Cat{MDN}$. \emph{Note that the diagram pictured above is technically not a legal wiring diagram}: typically multiple inputs cannot map into the same port, but we have used a shorthand here to represent the many input ports to each layer.}
\label{fig:cortex}
\end{figure}
It is now possible to take the tensor product of layers to form wider, parallel layers, or to wire multiple layers together to form a multilayer neural network as in Figure~\ref{fig:cortex}. This is done by defining a morphism $\text{V} = (\epsilon, \sigma_!) : (V_1, V_2, V_3) \to V$ from each box $V_i$ to a wrapper $V$, and the process may be continued ad infinitum.

In principle it should not be hard to prove that there is an additional operad of feedforward networks, denoted $\cat{O}_\Cat{FFN}$, as a suboperad of $\cat{O}_\Cat{MDN}$, with objects tensor products of neurons and morphisms feedforward networks. $\cat{O}_\Cat{FFN}$ constitutes a domain-specific language for writing neural networks, one that allows them to be easily composed in serial, in parallel, or by insertion.

\end{example}

\begin{example}[Plug and play]\label{example:cortex}
As promised, we can now ``plug in'' Example~\ref{ex:products}, Example~\ref{ex:behavior}, and Example~\ref{ex:nnet} to form an extended, mode-dependent model of the visual system, $\text{VS} = (\epsilon, \sigma) : (B, V) \to VS$. Essentially: light comes in through the eyes and activates the retinal nerve cells, whose signals are (1) used to define a blink reflex through the pons and (2) fed to a slice of the visual cortex, where the information is interpreted, analyzed, and fed as a $\{0,1\}$ output to some other region of the brain.


\begin{figure}[h!]
\centering
	\begin{tikzpicture}[wiring diagram,bb port sep=1, bb port length=2.5pt, bbx=.6cm, bb min width=.6cm, bby=1.3ex]
		\node[bbthick={2}{1}{thin}{very thick}, bb name=${E_1}$] (E1) {};
		\node[bbthick={2}{1}{thin}{very thick}, below= 3 of E1, bb name=${E_2}$] (E2) {};
		\node[bbthick={2}{2}{very thick}{thin}, below right = 0 and 2 of E1, bb name=${P}$] (P) {};
		\node[bbthick={2}{1}{very thick}{thin}, right= 4 of P, bb name=$V$] (V) {};
		\node[bb={2}{1}, fit={($(E2.south west)+(0,-2)$) (P) (V) ($(E1.north)+(0,2)$)},bb name =$VS$] (VS) {};
		\draw[ar, very thick] (E1_out1) to (P_in1);
		\draw[ar, very thick] (E2_out1) to (P_in2);
		\draw[ar] let \p1=(P.north east), \p2=(E1.north west), \n1={\y1+\bby}, \n2=\bbportlen in (P_out1) to[in=0] (\x1,\y2+\n2) -- (\x2-\n2,\y2+\n2) to[out=180] (E1_in1);
		\draw[ar] let \p1=(P.south east), \p2=(E2.south west), \n1={\y1+\bby}, \n2=\bbportlen in (P_out2) to[in=0] (\x1,\y2-\n2) -- (\x2-\n2,\y2-\n2) to[out=180] (E2_in2);
		\draw[ar] (VS_in1') to (E1_in2);
		\draw[ar] (VS_in2') to (E2_in1);
		\draw[ar, very thick] (E1_out1) to (V_in1);
		\draw[ar, very thick] (E2_out1) to (V_in2);
		\draw[ar] (V_out1) to (VS_out1');
		\draw[label] node[right=1.2pt of VS_out1]     {$\outp{VS}_a : \{0,1\}$};
		\draw[label] node[left=1.2pt of VS_in1]     {$\inp{VS}_a : [0,1]$};
		\draw[label] node[left=1.2pt of VS_in2]     {$\inp{VS}_b : [0,1]$};
	\end{tikzpicture}
\caption{A wiring diagram of the visual pathway.}
\label{fig:last}
\end{figure}
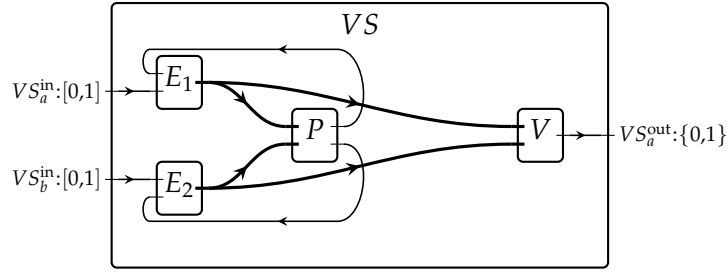

The entire system is reducible to the retinal nerves and cortical neurons. Recall the mode-dependent networks defined up to now: 
\begin{align*}
\text{Nerve} &: (M_N, N) \to R \\
\text{Eye} &: (R_1, ..., R_n) \to E \\
\text{Blink} &: (E_1, E_2, P) \to B \\
\text{Neuron} &: ((M_N, N), S) \to R \\
\text{V}_i &: (R_1, ..., R_n) \to V_i\\
\text{V} &: (V_1, V_2, V_3) \to V \\
\text{VS} &: (B, V) \to VS
\end{align*}
Then the previous statement, that the system is reducible to nerves and neurons, is just the statement of the composition: \[\text{VS} \circ (\text{V}, \text{Blink}) \circ (\text{Eye}, \text{Eye}, \text{V}_1, \text{V}_2, \text{V}_3) \circ (\text{Nerve}, ..., \text{Nerve}, \text{Neuron}, ..., \text{Neuron}).\]

\end{example}

\chapter*{Acknowledgments}

Thanks go to Eugene Lerman for many useful and interesting discussions, and for pointing out the work of Mike Field on asynchronous networks \cite{field}, which made this paper possible. We also appreciate many wonderful conversations with Dylan Rupel, who had the insight that internal states could be different than communicative modes. 

\end{document}